\documentclass[12pt]{article}

\usepackage{amsmath}%
\usepackage{amsfonts}%
\usepackage{amssymb}%
\usepackage{graphicx}
 
\newtheorem{theorem}{Theorem}

\newtheorem{proposition}{Proposition}

\numberwithin{equation}{section}

\newenvironment{proof}[1][Proof]{\noindent \textbf{#1.} }{\ \quad \rule{0.5em}{0.5em}}
\begin{document}

\title{A decomposition of  the bifractional Brownian motion and some applications}
\author{Pedro Lei, David Nualart\thanks{  D. Nualart is  supported    by
  the NSF grant   DMS0604207.}  \\
Department of Mathematics \\
University of Kansas \\
Lawrence, Kansas, 66045 USA}
\date{}
\maketitle

\begin{abstract}
In this paper we show a decomposition of the bifractional Brownian motion with parameters $H,K$ into the sum of a fractional Brownian motion with Hurst parameter $HK$ plus a stochastic process with absolutely continuous trajectories. Some applications of this decomposition are discussed.
\end{abstract}
\maketitle

 \section{Introduction}

The  \textit{bifractional Brownian motion} is a generalization of the fractional Brownian motion,   defined as a    centered Gaussian process $B^{H,K}=(  B_t^{H,K},  t \ge 0 )$,   with covariance
\begin{equation} \label{cov1}
R^{H,K}(t,s) =2^{-K} ((t^{2H} + s^{2H})^{K} - |t-s|^{2HK}),
\end{equation}
where $H \in (0,1)$ and $K \in (0,1]$. 
Note that, if $K=1$ then $B^{H,1}$ is a fractional Brownian motion with Hurst parameter $H \in (0,1)$, and we denote this process by $B^{H}$.
Some properties of the bifractional Brownian motion have been studied in \cite{CJ2003}, and \cite{FC2006}. In particular, in \cite{FC2006} the authors show that the bifractional Brownian motion behaves as a fractional Brownian motion with Hurst parameter $HK$.
The stochastic calculus with respect to the bifractional Brownian motion has been recently developed in the references  \cite{KRT} and \cite{ET}.

The purpose of this note is to show a decomposition of the bifractional Brownian motion as the sum of a fractional Brownian motion with Hurst parameter $HK$ plus a process with absolutely continuous trajectories. This decomposition  leads to a better understanding, and   simple proofs  of some of the properties of the bifractional Brownian motion that have been obtained in the literature.

\setcounter{equation}{0}

\section{Preliminaries}

 Suppose that   $B^{H,K}$ is a bifractional Brownian motion with covariance (\ref{cov1}).
The following properties have been proved in \cite{CJ2003} and summarized in \cite{FC2006}.
\begin{enumerate}
\item[(i)] The  bifractional Brownian motion with parameters $(H,K)$  is $HK$-self-similar, that is, for any $a>0$, the processes $( a^{-HK}B^{H,K}_{at}, t\ge 0 )$ and $( B^{H,K}_{t},  t\ge 0 )$ have
the same distribution. This is an immediate consequence of the fact that the covariance function is homogeneous of order $2HK$.
\item[(ii)]   For every $s,t \in [0,\infty)$, we have
\begin{equation}
2^{-K}|t-s|^{2HK} \le \mathbb{E}\left( (B^{H,K}_t - B^{H,K}_s)^2  \right) \le 2^{1-K}|t-s|^{2HK}.
\end{equation}
This inequality shows that the process $B^{H,K}$ is a quasi-helix in the sense of J.P.Kahane \cite{J1981, J1985}.  Applying  Kolmogorov's continuity criterion, it follows that  $B^{H,K}$ has a version with H\"older continuous  trajectories  of order $\delta$ for any $ \delta <HK$. 
 \end{enumerate}
Note that the bifractional Brownian motion does not have stationary increments, except in the case $K=1$.

It turns out that the bifractional Brownian motion is related to some stochastic partial differential equations. For example, suppose that $(u(t,x), t\ge 0, x\in \mathbb{R})$ is the solution of the one-dimensional stochastic heat equation on $\mathbb{R}$ with initial condition $u(0,x)=0$
\[
\frac {\partial u}{\partial t}= \frac 12 \frac{\partial^2 u}{\partial x^2} + \frac {\partial ^2W}{\partial t \partial x},
\]
where $W=\{ W(t,x), t\ge 0, x\in \mathbb{R} \}$ is a two-parameter Wiener process.  In other words, $W$ is a centered Gaussian process with covariance
\[
E(W(t,x)W(s,y))= (t\wedge s) (|x| \wedge |y|).
\]
Then, for any $x\in \mathbb{R}$, the process $(u(t,x), t\ge 0)$ is a bifractional Brownian motion with parameters $H=K=\frac 12$, multiplied by the constant $(2\pi)^{\frac 14} 2^{-\frac 18}$. In fact,
\[
u(t,x)= \int_0^t \int_\mathbb{R} p_{t-s}(x-y) W(ds,dy),
\]
where $p_{t}(x)=\frac{1}{\sqrt{ 2 \pi}} e^{-\frac{x^2}{2}}$, and the covariance of $u(t,x)$ is given by
\begin{eqnarray*}
\mathbb{E} (u(t,x)u(s,x))&=& \int_0^{t \wedge s} \int_\mathbb{R} p_{t-r}(x-y) p_{s-r}(x-y) dydr \\
&=&  \int_0^{t \wedge s}  p_{t+s-2r}(0)dr \\
&=& \frac 1{\sqrt{2\pi}} (\sqrt{t+s} -\sqrt{|t-s|}).
\end{eqnarray*}.

\setcounter{equation}{0}
\section{A decomposition of the bifractional Brownian motion}
 Consider the following  decomposition of the covariance function of the bifractional Brownian motion:
\begin{eqnarray}  \notag
R^{H,K}(t,s)& =& \frac{1}{2^K} [(t^{2H} + s^{2H})^K - t^{2HK} - s^{2HK}]   \\
&& + \frac{1}{2^{K}}[t^{2HK} + s^{2HK} - |t-s|^{2HK}].  \label{cov2}
\end{eqnarray}
The second summand in the
  above equation is the covariance of a fractional Brownian motion with Hurst parameter $HK$.   The first summand turns out to be non-positive definite and with a change of sign it will be the covariance of a   Gaussian process. In order to define this process, consider
  a    standard Brownian motion $(W_{\theta}, \theta \ge 0)$. For any $0<K <1$, define the process  $X^{K} = (  X^{K}_t, t\ge 0)$ by
\begin{equation}
  X^{K}_t=\int_{0}^{\infty}(1-e^{-\theta t}) \theta^{-\frac{1+K}{2}} dW_{\theta}.  \label{3.2}
\end{equation}
Then, 
$X^{K}$ is a centered Gaussian process with covariance:
\begin{eqnarray}
\gamma^{K}(t,s) = \mathbb{E} [X^{K}_t X^{K}_s]  \notag
&=&  \int_{0}^{\infty}(1-e^{-\theta t})(1-e^{-\theta s}) \theta^{-1-K} d\theta\\
&=& \frac{\Gamma(1-K)}{K}[t^{K}+s^{K}-(t+s)^{K}].  \label{cov}
\end{eqnarray}
In this way we obtain the following result.
\begin{proposition} \label{p1} 
Let $B^{H,K}$ be a bifractional Brownian motion, and suppose that $(W_{\theta}, \, \theta \ge 0)$ is a Brownian motion independent of $B^{H,K}$.   Let $X^K$ be the process defined in  (\ref{3.2}). Set $X^{H,K}_t= X^K_{t^{2H}}$.
Then,   the processes 
$(C_{1} X^{H,K} _t+B^{H,K}_t  , t\ge 0)$ and $(C_{2}B_t^{HK}, t\ge 0)$ have the same distribution, where
  $C_{1} =  \sqrt{\frac{2^{-K}K}{\Gamma(1-K)}}$ and $C_{2} = 2^{\frac{1-K}{2}}$. 
\end{proposition}

\begin{proof}
Let $Y_{t} = C_{1} X^{H,K}_t+B^{H,K}_t$. Then,  from (\ref{cov2}) and (\ref{cov})  for $s,t\ge 0$ we have
\begin{eqnarray*}
\mathbb{E}(Y_{s}Y_{t}) 
&=& C_{1}^{2} \mathbb{E}(X^{K}_{s^{2H}}X^{K}_{t^{2H}}) + \mathbb{E}(B^{H,K}_sB^{H,K}_t) \\
&=& \frac{1}{2^{K}}(t^{2HK}+s^{2HK}-(t^{2H}+s^{2H})^{K}) + \frac{1}{2^{K}} ((t^{2H} + s^{2H})^{K} - |t-s|^{2HK}) \\
&=& \frac{1}{2^{K}}(t^{2HK} + s^{2HK} - |t-s|^{2HK}),  
\end{eqnarray*}
which completes the proof.
\end{proof}

The next result provides some regularity properties for the process $X^K$.
\begin{theorem} \label{t1}
 The process $X^{K}$  has a version with trajectories which are infinitely differentiable trajectories on $(0,\infty)$ and absolutely continuous on $[0,\infty)$.   
 \end{theorem}

\begin{proof}
Note that $\mathbb{E}[(X^{K}_t)^{2}] = C_{3}t^{K}$, where $C_{3}$ = $\frac{\Gamma(1-K)}{K}(2-2^{K})$.
For any $t>0$, define $Y_{t} = \int_{0}^{\infty} \theta^{\frac{1-K}{2}} e^{- \theta t} \, dW_{\theta}$.  This integral exists because 
$$
\mathbb{E}[Y_{t}^{2}] = \int_{0}^{\infty}{ \theta^{1-K} e^{-2 \theta t} d\theta } = \Gamma(K) 2^{K-2} t^{K-2} .
$$
Applying Fubini's theorem and Cauchy-Schwartz inequality, we have:
\[
\mathbb{E}\left(\int_{0}^{t} |Y_{s}| ds\right)
= \sqrt{\frac 2 \pi} \int_{0}^{t}  \sqrt{\mathbb{E}[|Y_{s}|^{2}]} ds 
= \sqrt{ \Gamma(K)\frac 2 \pi}  2^{\frac K2-1}  \int_{0}^{t}  \, s^{\frac{K-2}{2}}ds
< \infty.
\]
 On the other hand, applying stochastic Fubini's theorem, we have:
\begin{eqnarray*}
\int_{0}^{t} Y_{s} ds
&=& \int_{0}^{t} (\int_{0}^{\infty} \theta^{\frac{1-K}{2}} e^{- \theta s} \, dW_{\theta} \,) ds \\
&=& \int_{0}^{\infty}  \theta^{\frac{1-K}{2}} (\int_{0}^{t}  e^{- \theta s} \,  ds \,) dW_{\theta} \\ 
&=& \int_{0}^{\infty}  \theta^{-\frac{1+K}{2}} ( 1 -e^{- \theta t}) \, dW_{\theta} \\
&=& X^{K}(t).
\end{eqnarray*}
This implies that $X^{K}$ is absolutely continuous and $Y_{t} = (X_t^{K})' $ on $(0,\infty)$. Similarly, the $n^{th}$ derivative of $X^{K}$  exists on $(0,\infty)$  and it is  given by 
\[
(X^K_t )^{(n)}= \int_{0}^{\infty} (-1)^{n-1}(\theta)^{n-\frac{1}{2}-\frac{K}{2}} e^{-\theta t} \, dW_{\theta}.
\]
\end{proof}

The next proposition provides some information about the behavior of $X^K$ at the origin.
 
\begin{proposition}
There exists a nonnegative random variable $G(\omega)$ such that for all $t\in [0,1]$
\[
|X^K_t| \le G(\omega)  \sqrt{t^K \log\log (t^{-1})}.
\]
\end{proposition}

\begin{proof}
Applying an integration by parts  yields
\[
X^K_t = \int_0^\infty \varphi(\theta,t) W_\theta d\theta, \\
\]
where
\[
\varphi(\theta,t)=t e^{-\theta t} \theta^{-\frac{1+K}{2}}-\frac{1+K}{2}\theta^{-\frac{3+K}{2}}(1-e^{-\theta t}). 
\]
By the law of iterated logarithm for the Brownian motion, given $c>1$ we can find
two random points $0<t_0 <t_1$, with $t_0 < \frac 1e, t_1 >e$,  such that almost surely, for all $\theta\le t_0$ 
\[
|W_\theta| \le   c \sqrt{2\theta \log \log \theta^{-1}}
\]
and for all $\theta \ge t_1$
\[
|W_\theta| \le   c \sqrt{2\theta \log \log \theta }. 
\]
Then we make the decomposition
\[
X^K_t= \int_0^{t_0} W_\theta \varphi(\theta,t) d\theta + \int_{t_0}^{t_1} W_\theta \varphi(\theta,t) d\theta +\int_{t_1}^\infty W_\theta \varphi(\theta,t) d\theta.
\]
For the first term we obtain
\[
\left|\int_0^a W_\theta \varphi(\theta,t) d\theta   \right|
\le c \int_0^{t_0}   \sqrt{2\theta \log \log \theta^{-1}} |\varphi(\theta,t)| d\theta
\le G_1 t,
\]
for some  nonnegative random variable $G_1$. Similarly, we can show  that
\[
\left
|\int_{t_0}^{t_1} W_\theta \varphi(\theta,t) d\theta  \right| \le G_2 t.
\]
With the change of variables $\eta = \theta t$ the third term can be bounded as follows
\begin{eqnarray*}
\left| \int_{t_1} ^\infty  W_\theta \varphi(\theta,t) d\theta   \right|
\le  c \int_{t_1}^\infty   \sqrt{2\theta \log \log \theta }  |\varphi (\theta, t)| d\theta \\
\le 2\sqrt{2} c  t^{\frac{K}{2}} \int_{t_1}^\infty \sqrt{\log \log \frac{\eta}{t}} ( \eta^{-1-\frac{K}{2}}   + \eta^{-2- \frac K2}) d\eta.
\end{eqnarray*}
Applying the inequality
\[
\log (\log |\eta| + \log t^{-1}) \le \log 2 + \log |\log \eta| + \log \log t^{-1}
\] 
we obtain
\[
\left| \int_{t_1} ^\infty  W_\theta \varphi(\theta,t) d\theta   \right|
\le  G_3  t^{\frac{K}{2}} \sqrt{\log \log t^{-1}}.
\]
This completes the proof. 
\end{proof}

\setcounter{equation}{0}
\section{Applications}
We first describe the space of integrable functions with respect to the bifractional Brownian motion.

Suppose that $X=(X_t, t\in [0,T])$ is a continuous zero mean Gaussian process.  Denote by $\mathcal{E}$ the set of step functions on $[0,T]$. Let $\mathcal{H}_{X}$ be the Hillbert space defined as the closure of $\mathcal{E}$ with respect to the scalar product
\[
\langle \textbf{1}_{[0,t]},\textbf{1}_{[0,s]} \rangle_{\mathcal{H}}=\mathbb{E}(X_tX_s).
\]
The mapping $\textbf{1}_{[0,t]} \rightarrow X_t$ can be extend to a linear isometry between $\mathcal{H}_{X}$ and the Gaussian space $H_{1}(X)$ associated with $X$. We will denote this isometry by $\varphi$ $\to$ $X(\varphi)$. The problem is to find $\mathcal{H}_{X}$ for a particular process $X$.

In the case of the standard Brownian motion $B$, the space $\mathcal{H}_{B}$ is $L^2([0,T])$. For the fractional Brownian motion $B^H$ with Hurst parameter $H\in (0,\frac 12)$ it is known (see  \cite{DU}) that $\mathcal{H}_{B^H}$ coincides with the fractional Sobolev space $I^{\frac 12-H}_{0+}(L^2([0,T]))$. In the case $H >\frac 12 $, the space $\mathcal{H}_{B^H}$ contains distributions, according to the work by Pipiras and Taqqu 
\cite{PT}.  In a recent work,   Jolis \cite{Jo} has proved that if $H>\frac 12$, the space $\mathcal{H}_{B^H}$  is the set of restrictions to the space of smooth functions  $\mathcal{D}(0,T)$ of the distributions of $W^{1/2-H,2}(\mathbb{R})$ with support contained in $[0,T]$.

For the bifractional Brownian motion we can prove the following result. 
As before, we denote by  $X^{H,K}_t$ the process $X^K_{t^{2H}}$.

\begin{proposition} \label{p2}
For $H \in (0,1)$ and $K \in (0,1]$, the equality $\mathcal{H}_{X^{H,K}} \cap \mathcal{H}_{B^{H,K}} = \mathcal{H}_{B^{HK}}$ holds.
\end{proposition}

\begin{proof}
For any step function $\varphi \in \mathcal{E}$   the following equality  is a consequence of the decomposition proved in Proposition \ref{p1}  and the independence of $X^{K}$ and $B^{H,K}$:
\begin{equation*}
C_{1}\mathbb{E}(| \int_0^T \varphi (t) dX_{t }^{H, K}|^{2}) + \mathbb{E}(| \int _0^T\varphi (t) dB_{t}^{H,K}|^{2}) = C_{2} \mathbb{E}(| \int_0^T \varphi (t) dB_{t}^{HK}|^{2}) \\
\end{equation*}
where $C_{1}$ and $C_{2}$ are positive constants. 
The equality $\mathcal{H}_{X^{H,K}} \cap \mathcal{H}_{B^{H,K}} = \mathcal{H}_{B^{HK}}$ follows immediately. 
\end{proof}

On the other hand,    for any step function  $\varphi  \in \mathcal{E}$ it holds that
\[ 
\mathbb{E}(X^{H,K} (\varphi)^2)\le  C_{H,K}  \left( \int_0^T |\varphi(t)| t^{HK-1} dt \right)^2,
\]
where $C_{H,K}$ is a constant depending only on $H$ and $K$. As a  consequence,  $L^1([0,T]; t^{HK-1} dt) \subset \mathcal{H}_{X^{H,K}}$.
 The proof is sketched as follows. By taking partial derivative of the covariance function $\gamma^{K}$ given in  (\ref{cov}) it follows that
\[
\frac{ \partial^2 \gamma^{K}   (s^{2H}, t^{2H})}{\partial s  \partial t} = C_{H,K} (t^{2H}+s^{2H})^{K-2} t^{2H-1} s^{2H-1}.
\]
for some constant $C_{H,K}$.
Then, for any  $\varphi \in \mathcal{E}$ it holds that
\begin{eqnarray*}
&& \int_0^T \int_0^T |\varphi(s) \varphi(t)| (st)^{2H-1} (t^{2H}+s^{2H})^{K-2} ds dt \\
&& \le  \int_0^T \int_0^T |\varphi(s) \varphi(t)| (st)^{2H-1} (s^{2H} t^{2H})^{\frac{K-2}{2}} ds dt \\
&& = \left( \int_0^T |\varphi(t)| t^{HK-1} dt \right)^2.
\end{eqnarray*}
By H\"older's inequality this implies that $L^p([0,T]) \subset \mathcal{H}_{X^{H,K}}$ for any $p>\frac 1 {HK}$. As a consequence, 
\[
L^1([0,T]; t^{HK-1} dt) \cap \mathcal{H}_{B^{H,K}} \subset \mathcal{H}_{X^{H,K}} \cap \mathcal{H}_{B^{H,K}} = \mathcal{H}_{B^{HK}}.
\]
In the case $HK<\frac 12$ this implies that a function in   $\mathcal{H}_{B^{H,K}}$ 
which is in  $L^1([0,T]; t^{HK-1} dt) $  must belong to the Sobolev space  $I^{\frac{1}{2}-HK}_{0^{+}}(L^2([0,T] ))$.

Consider now the   notion  of  $\alpha$-variations for a continuous process $X=(X_t, t\ge 0)$. The process $X$ admits an $\alpha$-\textit{variation}   if 
\begin{eqnarray}  
  V_{t}^{n, \alpha}(X)=\sum_{i=0}^{n-1} | \Delta X_{t_{i}}|^{\alpha} .  \label{var1}
\end{eqnarray} 
converges in probability as $n$ tends to infinity for all $t\ge 0$, where 
  $ t_i= \frac{it}n $   and  $\Delta X_{t_{i}}=X_{t_{i+1}}-X_{t_{i}}$.

 As a consequence of the Ergodic Theorem and the scaling property of the fractional Brownian motion, it is easy to show (see, for instance, \cite{Ro}) that the fractional Brownian motion with Hurst parameter $H\in (0,1)$ has an   $\frac 1H $-\textit{variation}  equals to $C_H t$, where $C_H = \mathbb{E}(|\xi|^H)$ and $\xi$ is a standard normal random variable. Then,  Proposition \ref{p1}  allows us to obtain the  $\frac 1{HK} $-\textit{variation} of bifractional Brownian motion. This provides a simple proof of a similar result in \cite{FC2006}.
  
\begin{proposition}
The bifractional Brownian motion with parameters $H$ and $K$ has a  $\frac 1{HK} $-\textit{variation} equals to $C_2^\frac{1}{HK} C_{HK}t$, where $C_{HK}=  \mathbb{E}(|\xi|^{HK})$ and $\xi$ is a standard normal random variable.
\end{proposition}

\begin{proof}
Proposition \ref{p1} implies $B^{H,K}_t = C_2 B^{HK} _t - C_1 X^ {H,K}_t$. \\
Applying Minkowski's inequality, 
\[
\left( \sum_{i=0}^{n-1} | \Delta B_{t_{i}}^{H,K}|^{\frac{1}{HK}} \right)^{HK} \le 
C_2 \left( \sum_{i=0}^{n-1} | \Delta B_{t_{i}}^{HK}|^{\frac{1}{HK}} \right)^{HK} + C_1 \left( \sum_{i=0}^{n-1} | \Delta X_{t_{i} }^{H,K}|^{\frac{1}{HK}} \right)^{HK}.
\]
On the other hand, 
\[
C_2 \left( \sum_{i=0}^{n-1} | \Delta B_{t_{i}}^{HK}|^{\frac{1}{HK}} \right)^{HK} - C_1 \left( \sum_{i=0}^{n-1} | \Delta X_{t_{i} }^{H,K}|^{\frac{1}{HK}} \right)^{HK} \le
\left( \sum_{i=0}^{n-1} | \Delta B_{t_{i}}^{H,K}|^{\frac{1}{HK}} \right)^{HK}.
\]
 From the results for the fractional Brownian motion we know that
 \[
\lim_{n \to \infty} V_t^{n,\frac{1}{HK}} (B^{HK}) = C_{HK} t
\]
 almost surely and in $L^1$. To complete the proof, it is enough to show that $  \sum_{i=0}^{n-1} | \Delta X_{t_{i}  }^{H,K}|^{\frac{1}{HK}}  $ converges to zero. We can write
\[
\sum_{i=0}^{n-1} | \Delta X_{t_{i} }^{H,K}|^{\frac{1}{HK}}
\le    \sup_{i}  | \Delta X^{H,K} _{t_i}  |^{\frac{1}{HK}-1}  \sum_{i=0}^{n-1} | \Delta X_{t_{i} }^{H,K}|.
\]
The first factor in the above expression converges to zero by continuity, and the second factor is bounded by the total variation of $X^{H,K}$ on $[0,T]$ since $X^{H,K}$ is absolutely continuous on $[0,T]$ by Theorem \ref{t1}. The proof is complete.
 \end{proof}
 
Similarly,  for the    $\frac{1}{HK}$-\textit{strong variation} of the process $B^{H,K}$  
we can show that  
\[
\lim_{\varepsilon \to 0} \frac{1}{\varepsilon} \int_{0}^{t} |B^{H,K}_{s+\varepsilon} - B^{H,K}_{s}|^{\frac{1}{HK}} ds = C_2^{\frac{1}{HK}} C_{HK} t.
\]

\medskip
In  \cite{TX} the authors have proved the 
 Chung's law of the iterated logarithm   for the bifractional Brownian motion:
 \begin{equation}
\liminf_{r \to 0} \frac {\max_{t \in [0,r]} |B_{t+t_0}^{H,K}-B_{t_0}^{H,K}|} {r^{HK} / (\log \log (1/r))^{HK}} = C_0(HK),
\end{equation}
where $C_0$ is a positive and finite constant depneding on HK, for all $t_0 \ge 0$. 
A similar result for the fractional Brownian motion was obtained by Monrad and  Root\'zen in \cite{MR}. 
The decomposition obtained in this paper allows us deduce   Chung's law of iterated logarithm  for $t_0>0$ for the bifractional Brownian motion, from the same result for the fractional Brownian motion with Hurst parameter $HK$, with the same constant.   

Let us finally remark that the decomposition established in this paper permits to develop a stochastic calculus for the bifractional Brownian motion using the well-known results in the literature on the stochastic integration with respect to the fractional Brownian motion, and taking into account that the process $X^{H,K}$ has absolutely continuous trajectories.

\end{document}